\documentclass[a4paper,12pt]{article}
\usepackage{amsfonts, amsmath, amsthm, amssymb}

\theoremstyle{plain}
\newtheorem{theorem}{Theorem}[section]
\newtheorem{lemma}[theorem]{Lemma}
\newtheorem{proposition}[theorem]{Proposition}
\newtheorem{corollary}[theorem]{Corollary}

\theoremstyle{definition}
\newtheorem{definition}[theorem]{Definition}
\newtheorem{example}[theorem]{Example}

\theoremstyle{remark}
\newtheorem{remark}[theorem]{Remark}

\advance\textheight by 1.2in
\theoremstyle{plain}

\newcommand{\Sym}{\ensuremath \mathrm{Sym}}
\newcommand{\ord}{\ensuremath \mathop \mathrm{ord} \nolimits}
\newcommand{\Obst}{\ensuremath \mathop \mathrm{Obst} \nolimits}

\author{Ekaterina Shemyakova and Franz Winkler\\\\
 Research Institute for Symbolic Computation (RISC),\\
J.Kepler University,\\
Altenbergerstr. 69, A-4040 Linz, Austria\\
 \{kath, Franz.Winkler\}@risc.uni-linz.ac.at\\
  }
\title{Obstacles to the Factorization of Linear Partial Differential Operators into
Several Factors.}
\begin{document}
\maketitle
\date
\begin{abstract}
We consider algorithms for the factorization of linear partial
differential operators. We introduce several new theoretical
notions in order to simplify such considerations.
We define an obstacle and a
ring of obstacles to factorizations.
We derive some interesting
facts about the new objects, for instance, that they are invariant under
gauge transformations.
An important theorem for the construction of factoring algorithms is proved: a
factorization is defined uniquely from a certain moment on.

For operators of orders three and two, obstacles are found
explicitly.

* This work was supported by Austrian
Science Foundation (FWF) under the project SFB F013/F1304.
\end{abstract}
\section{Introduction}

We investigate the problem of the factorization of a linear partial
differential operator over some field. The starting point is the
algorithm of Grigoriev--Schwarz \cite{GS}, which extends a
factorization (into coprime factors) of the symbol of an operator to
a factorization of this operator. At the first step of the algorithm,
only the highest terms of the factors of a factorization are known.
At every succeeding step, either we determine the next component in each
factor, or we conclude that there is no such factorization. In the
latter case we lose all of the information about the operator that we
obtained implicitly during the execution of the algorithm. Here, we
suggest that the information can be used, and introduce the notions of partial
factorizations and common obstacles. The latter are exactly the
invariants of Laplace for a second-order strictly hyperbolic
operator \cite{tsarev05}.

We use partial factorizations to prove Theorem \ref{GS_non_coprime} below,
which states that a factorization is uniquely defined from
a certain moment on. The theorem of Grigoriev-Schwarz \cite{GS} is a
particular case of this theorem.

Neither a common obstacle, nor its symbol is invariant or unique in
general; some examples can be seen in \cite{kart}.
However, we
consider the factor ring of the ring of polynomials (corresponding
to the ring of linear differential operators) modulo certain
homogeneous ideal. We call it the ring of obstacles. The symbols of
all common obstacles belong to the same class of this factor ring!
We call the class the obstacle to factorizations. It is significant
that the obstacle is invariant under gauge transformations.
Others interesting properties are also derived.

We compute explicit formulas for common obstacles for bivariate
second- and third-order operators.

It is important to mention that the theory has an application: it
was used to obtain a full system of invariants for bivariate
hyperbolic third-order operator \cite{invariants}.

  The present work is an (independent) continuation of the work \cite{spain},
where we
consider factorizations into two factors.
\section{Preliminaries}

 Let $K$ be some field.
 Let $\Delta=\{ \partial_1, \dots, \partial_n\}$
be commuting derivations acting on $K$.
 Consider the ring of linear differential operators
\[
K[D]=K[D_1, \dots, D_n]\ ,
\]
where $D_1, \dots, D_n$ correspond to the derivations $\partial_1,
\dots, \partial_n$ respectively.

The totality of all linear differential operators of orders
$\leqslant i$ with defined left and right multiplication is a
$K$-bimodule, which we denote by $K_{ \leqslant i}$. Thus we have the
filtration $ \dots \supset K_{ \leqslant i} \supset K_{ \leqslant i-1} \supset
\dots \supset K_{ \leqslant 0}.$ Consider the associate algebra
\[
Smbl_*=\sum_{i \geq 0}Smbl_i, \quad Smbl_i=K_{ \leqslant i} \diagup K_{
\leqslant i-1}.
\]
$K$-module $Smbl_*$ is a commutative $K$-algebra, which is
isomorphic to the ring of polynomials $K[X]=K[X_1, \dots, X_n]$ in
$n$ variables. The image of the operator $L \in K[D]$ by the natural
projection is some element $\Sym_L$ of $K[X]$. Actually, the symbol
of an operator is a homogeneous polynomial corresponding to the sum
of the highest terms.

 We use the notation
\[
D^{(i_1, \dots, i_n)}:=D_1^{i_1} \dots D_n^{i_n}\ ,
\]
and define the order as follows:
\[
|D^{(i_1, \dots, i_n)}|=\ord(D^{(i_1, \dots, i_n)}):= i_1 + \dots +
i_n\ ,
\]
and in addition the order of the zero operator is $-\infty$.

 For a homogeneous polynomial $S \in K[X]$ we define the operator
 $\widehat{S }\in K[D]$, which is the result of the substitution of
 the operator $D_i$ for
 each variable $X_i$. If there is no danger of misunderstanding we
 use just $S$ to denote the operator $\widehat{S}$.  By $K_i[D]$ we denote the set of
 all operators in $K[D]$ of order $i$.

 Now every operator $L \in K[D]$ can be written as
\begin{equation}
 L = \sum_{|J| \leq d} a_J D^J\ =\sum_{i=0}^{d} L_{i} \ ,
\end{equation}
where $a_J \in K, \ J \in \bf{N}^n$, and $L_i$ is the $i$th
component of $L$.
\section{Partial Factorizations}

We start by generalizing several notions that were
introduced for the case of factorizations into two
factors in \cite{spain}.
\begin{definition}
Let $L \in K[D]$ and suppose that its symbol has a decomposition
$\Sym_L=S_1 \dots S_k$. Then we say that the factorization
\[
L=F_1 \circ \dots  \circ F_k , \quad \mbox{where} \quad
\Sym_{F_i}=S_i\ , \ \forall i \in \{1, \dots, k\},
\]
is \emph{of the factorization type} $(S_1)(S_2)\dots(S_k)$.
\end{definition}
\begin{definition} \label{partial_fact_k_factors}
Let for some operators $L$, for $F_i \in K[D], \ i=1, \dots, k$ and for some
$t \in \{0, \dots,\ord(L) \}$
\begin{equation} \label{partial_fac_arb}
\ord(L-F_1 \circ \dots \circ F_k) < t
\end{equation}
holds. Then we say that $F_1 \circ \dots \circ F_k $ is \emph{a
partial factorization} of order $t$ of the operator $L$. If in
addition $S_i=\Sym_{F_i}, \ i=1, \dots, k$ (so $\Sym_L=S_1 \dots
S_k$), then this partial factorization is of the factorization type
$(S_1) \dots (S_k)$.
\end{definition}
\begin{remark} Every usual factorization of $L \in K[D]$ is a partial factorization
of order $0$.
\end{remark}
\begin{remark} Let $L \in K[D]$, $\ord(L)=d$. Then for every
 factorization of the symbol $\Sym_L=S_1 \dots S_k$
 the corresponding composition of operators $\widehat{S}_1 \circ \dots  \circ
 \widehat{S}_k$ is a partial factorization of order $d$.
\end{remark}

 Let $L \in K[D]$ and $F_1 \circ \dots \circ F_k$ be a partial
 factorization of order $t$. Note that the condition (\ref{partial_fac_arb})
 still holds if we change any term whose order is less than or
 equal to $t-(d-d_j)$ in any factor $F_j, \ j \in \{1, \dots , k \}$.
 Thus we obtain new partial factorizations of order less than or
 equal $t$. Thus we introduce the following definition.
\begin{definition} \label{part_fact_def_k_factors}
Let $L \in K[D]$, $\Sym_L=S_1 \dots S_k$, $\ord(S_i)=d_i,\ i=1, \dots, k$ and
\[
F_1 \circ \dots \circ F_k, \quad F'_1 \circ \dots \circ F'_k
\]
be partial factorizations of orders $t$ and $t'$ respectively. Let
$t'<t$, then $F'_1 \circ \dots \circ F'_k$ is an extension of $F_1
\circ \dots \circ F_k$ if
\[
\ord(F_i-F'_i)<t-(d-d_i), \ \forall i \in \{1, \dots, k \}\ .
\]
\end{definition}
\begin{example} Consider the fifth-order operator
\[
L= (D_1^2+D_2+1) \circ (D_1^2D_2 + D_1 D_2 + D_1 +1)\ .
\]
Compositions of the type
\[
(D_1^2+ \dots ) \circ (D_1^2D_2 + \dots) \ ,
\]
where ellipses mean arbitrarily chosen terms of lower orders, are partial
factorizations of order $5$. Their extensions are the following
fourth-order partial factorizations of the type
\[
(D_1^2+D_2 + \dots) \circ (D_1^2D_2 + D_1 D_2 + \dots).
\]
\end{example}
\begin{remark} Let $L \in K[D]$. Then $F_1 \circ \dots \circ F_k$ is
 a partial factorization of $L$ of the type $(S_1) \dots (S_k)$ if
 and only if $F_1 \circ \dots \circ F_k$
 is an extension of a partial factorization $S_1 \circ  \dots  \circ  S_2$.
\end{remark}
 Now we formulate two easy to prove facts, which will be useful for
 the proof of a theorem below.
\begin{proposition} \label{S1S2_coprime_unique}  Let $S_1$, $S_2$, $p $
 be homogeneous polynomials, in an arbitrary number of variables, of
 orders $d_1$, $d_2$, $s \ (0<s<d_1+d_2)$  respectively. Let, in
 addition, $S_1$ and $S_2$ be coprime.
 Then there exists at most one pair $(u,v)$ of homogeneous polynomials
 $u$ and $v$ of orders $s-d_1$ and $s-d_2$ respectively, such that
\[
  S_1 \cdot u + S_2  \cdot v = p.
\]
\end{proposition}
The second fact is the generalization of Proposition
\ref{S1S2_coprime_unique} to the case of non-coprime polynomials.
\begin{proposition} \label{S1S2_coprime_unique_gen}
Let $S_1$, $S_2$, $p $
 be homogeneous polynomials of orders $d_1$, $d_2$, $s$ respectively and
 in an arbitrary number of variables. Consider the polynomial $S_0$ of order
 $d_0$,
 which is the greatest common divisor of $S_1$ and $S_2$, and $ \ 0<s<d_1+d_2-d_0$.
 Then there exist at most one pair $(u,v)$ of homogeneous polynomials
 $u$ and $v$ of orders  $s-d_1$ and $s-d_2$ respectively, such that
\begin{equation} \label{gen_franz's_lemma}
  S_1 \cdot u + S_2  \cdot v = p.
\end{equation}
\end{proposition}
For every factorization $S_1 \cdot S_2$ of the symbol, the
corresponding composition of operators $\widehat{S}_1 \circ
\widehat{S}_2$ is a partial factorization  of the operator $L$. In
the case of coprime $S_1$ and $S_2$ there exists at most one
extension of this partial factorization to a factorization of $L$
\cite{GS}.
However, if there exists a nontrivial common divisor of $S_1$
and $S_2$, then this is not necessarily the case. Consider, for
example, the operator of Blumberg-Landau \cite{blumberg}:
\[
L=D_x^3+x D_x^2D_y+2 D_x^2+(2x+2)D_xD_y+D_x+(2+x)D_y,
\]
which is a frequently cited instance of an operator that has two
factorizations into different numbers of irreducible factors (that is,
factors that cannot be factored into factors of smaller orders):
\[
L=(D_x+1) \circ (D_x+1) \circ (D_x+xD_y) \ = \ (D_x^2+xD_xD_y + D_x+
(2+x) D_y) \circ (D_x+1).
\]
The same operator $L$ (the symbol of $L$ is $X^3+ x X^2Y$) has a
whole family of factorizations into two factors with the symbols
$S_1=X$ and $S_2=X(X+XY)$ respectively:
\[
L=\Big(D_x + 1 + \frac{1}{x+f_1(y)} \Big) \circ \Big(D^2_x + x D_x
D_y + (1- \frac{1}{x+f_1(y)})D_x + (x+1- \frac{x}{x+f_1(y)})D_y
\Big),
\]
where $f_1(y) \in K$ is a functional parameter.
 Thus, there is no uniqueness of factorization in this case. Nevertheless,
 even in the case of non-coprime symbols of factors we may formulate
 the following:
 \begin{theorem} \label{GS_non_coprime}
 Let $L \in K[D]$, $\Sym_L=S_1 \cdot S_2$, $\ord(L)=d$,
 and let the greatest common divisor of $S_1$ and $S_2$ be a
homogeneous polynomial $S_0$ of order $s$. Then for every
$(d-d_0)$th order partial factorization of the type $(S_1)(S_2)$,
there is at most one extension to a complete factorization of $L$ of
the same type.
\end{theorem}
 To prove the theorem, it is enough to prove the following lemma.
\begin{lemma}
 Let $L \in K[D]$, $\Sym_L=S_1 \cdot S_2$, $\ord(L)=d$,
 and let the greatest common divisor of $S_1$ and $S_2$
 be a homogeneous polynomial $S_0$ of order $s$.
 Then for every $t$-th ($t\leq  (d-d_0)$) order partial factorization of the
 type $(S_1)(S_2)$, there is at most one (up to lower order
 terms) extension to the partial factorization of order $t-1$ of the same type.
\end{lemma}
\begin{proof} If $d_0=0$, then the statement of the lemma is implies
 from  \cite{GS}.
 If $d_0>0$, consider the general form of a complete factorization of $L$,
 that extends the given $t$-th order partial factorization:
\begin{equation}\label{L=L1L2}
L= \left(\widehat{S}_1+ \sum_{j=0}^{k_1-1} G_j \right) \circ
\left(\widehat{S}_2 + \sum_{j=0}^{k_2-1} H_j \right)\ ,
\end{equation}
 where $k_1=\ord(S_1)$, $k_2=\ord(S_2)$ and $ G_j \in K_j[D], \quad H_i
 \in K_i[D], \quad j=0, \dots, (k_1-1), \ i=0, \dots, (k_2-1)$.
 By comparing components of order $t-1$ on the both sides of the equality
 (\ref{L=L1L2}), we get
\begin{equation}
\label{L_k+l-i}
    L_{t-1}= H_{t-k_1-1} \cdot S_1  +  G_{t-k_2-1} \cdot S_2 + P_{t-1},
\end{equation}
 where $P_{t-1}$ is a homogeneous polynomial of order $t$,
 which is determined uniquely
 by the polynomials
 $G_i$, $H_j$, $ i>t-k_1-1$, $ j>t-k_2-1$,
 which are components of order $t$ of the given partial
 factorization. We assume that polynomials $G_i$, $H_i$ are $0$ for $i<0$.

 Now, since $t-1 < d-d_0$, we may apply proposition
 \ref{S1S2_coprime_unique_gen}: there is at most one solution
 of equation (\ref{L_k+l-i}). Thus there exists at most one
 extension to a partial factorization of order $t-1$.
\end{proof}
\begin{corollary} Let $L \in K[X], \ \Sym_L=S_1 \cdot S_2,$ and $S_1$ and
$S_2$ be coprime. Then there is at most one common factorization of
$L$ of the type $(S_1)(S_2)$. Thus the theorem is a generalization
of the Grigoriev-Schwarz result \cite{GS}.
\end{corollary}
\begin{corollary} In the case of ordinary differential operators,
 the greatest common divisor of $S_1$ and $S_2$ is
\[
\gcd (S_1,S_2)=X^{d_0}, \quad \text{where} \quad
d_0=min(\ord(S_1),\ord(S_2)).
\]
 Then for every partial factorization of order
\[
max \big( \ord(S_1),\ord(S_2) \big)-1
\]
 there is at most one extension to a complete factorization.
\end{corollary}
\begin{corollary} \label{unique_partial_coprime}
 Let $L \in K[D]$, $\Sym_L=S_1 \cdot S_2$,
 and let $S_1$ be coprime with $S_2$. Then for every $t$, $\ t<\ord(L)$
 there is at most one (up to lower order terms)
 partial factorization of order $t$.
\end{corollary}
%
\section{Ring of Obstacles, Obstacles}
By induction on the number of factors and with the theorem of
uniqueness, if a factorization has the base \cite{GS}, one may prove
the following theorem:
\begin{theorem} \label{GS_unique_k_factors}
Let
 $L \in K[D]$, $\Sym_L=S_1 \cdot S_2 \dots S_k$ and  let $S_1, \dots, S_k$ be coprime.
 Then there exists at most one factorization of the type $(S_1)(S_2)\dots (S_k)$.
\end{theorem}
Now, consider factorable operators as a subvariety of all the
operators in $K[D]$ that have some fixed decomposition of the
symbol.
\begin{theorem} \label{dimension}
 Consider the variety of all the operators in $K[D]$
 that have the symbol $\Sym=S_1 \dots  S_k$, $\ord(S_i)=d_i$, $\ i=1, \dots, k$.
 Then the codimension of the subvariety of the operators that have a factorization
 of the type $(S_1)(S_2) \dots (S_k)$ equals
\[
\binom{n+d-1}{n}- \sum_{i=1}^k \binom{n+d_i-1}{n}\ .
\]
\end{theorem}
\begin{proof}
 Consider the problem of the factorization of $L$
 of the type $(S_1)(S_2) \dots (S_k)$ in the general form:
\begin{equation} \label{L=  ( S_1+}
L=  \Big( S_1+ \sum_{i=0}^{d_1-1}G_i^1 \Big) \circ  \dots  \circ
\Big( S_k+ \sum_{i=0}^{d_k-1}G_i^k \Big),
\end{equation}
 where $G_i^j$ denotes the $i$-th component in the $j$-th factor.
 Compare components of orders $t$, $\ 0 \leq t \leq \ord(L)-1$ on both sides
of (\ref{L=  ( S_1+}), then we have
\begin{equation} \label{P_t=(Sym_L/S_1)}
P_t=(\Sym/S_1) \cdot G_{t-d+d_1}^{1} + \dots + (\Sym/S_k) \cdot
G_{t-d+d_k}^{k},
\end{equation}
 where $P_{t}$ is a homogeneous polynomial of order $t$, which is
 determined uniquely by the polynomials $G_i$, $H_j$, $\ i>t-k_1$, $\ j>t-k_2$,
 and so it is known if we solve equations
 (\ref{P_t=(Sym_L/S_1)}) in ``descent'' order, that is if we start with
 $t=\ord(L)-1$, and reduce $t$ by one at each succeeding step.

 Polynomials $G_i, H_j, \ i>t-k_1, \ j>t-k_2$, and so $P_{t}$
are determined uniquely: it is an immediate consequence of the
following lemma:
\begin{lemma} \label{main}
 Let $S_1, \dots, S_k$  are pairwise coprime homogeneous polynomials of
 orders $d_1, \dots,d_k$ respectively.
 Denote $S=S_1 \dots S_k.$
 Then there is at most one tuple $(A_1, \dots, A_k)$ such that
\begin{equation} \label{P_t_S_i}
P_t=(S/S_1) \cdot A_{1} + \dots + (S/S_k) \cdot A_{k},
\end{equation}
where $\ord(P_t)=t, \ t< \ord(S),$ and $ \ord(A_i)+ \ord(S/S_i)=t.$
\end{lemma}
\begin{proof}
 Assume we have two such tuples: $(A'_1, \dots, A'_k)$ and $(A''_1, \dots, A''_k)$.
 Consider the difference of the equations corresponding to them, so
 we have
\begin{equation} \label{01}
0=(S/S_1) \cdot B_{1} + \dots + (S/S_k) \cdot B_{k},
\end{equation}
 where $B_i=A'_i-A''_i$, $i=1, \dots, k$.
  Without loss of generality
 we may assume $B_1 \neq 0$ and rewrite equation (\ref{01}) in the form
\[
-(S/S_1) \cdot B_{1}= (S/S_2) \cdot B_{2} + \dots + (S/S_k) \cdot
B_{k}.
\]
 Every component on the right is divisible by $S_1,$
 while $(S/S_1)$ is not so.
 Thus, $B_{1}$ is divisible by $S_1$,
 and so $\ord(B_1) \geq \ord(S_1).$

 On the other hand, we have $\ord(A_i)+\ord(S/S_i)=t$
 and $t< \ord(S),$ that is $\ord(A_i)< \ord(S_i),$ and so $ \ord(B_1)< \ord(S_i)$.
 This is a contradiction with the results of the previous paragraph.
\end{proof}

 The factorization exists if the system of all the equations (\ref{P_t=(Sym_L/S_1)}), $t=d-1, \dots, 0$
 is compatible.
 The codimension equals the number of independent equations
 in the coefficients of the operator.

 For every $t$ we have the linear equation (\ref{P_t=(Sym_L/S_1)})
 in the polynomials $G_{t-d+d_1}^{1}, \dots, G_{t-d+d_k}^{k}$,
 which is equivalent to the system of linear equations in their coefficients.
 Let the system be  $A \cdot \vec{g} =\vec{c}$, where $A$ is the matrix of the system.
 The system has a unique solution,
 and so the rank of the matrix $A$ equals the number $v$ of variables.
That is the columns of the matrix $A$ are linearly independent.

 The system $A \cdot \vec{g} =\vec{c}$ is compatible
 when vector $\vec{c}$ belongs to a $v$-dimensional affine space,
 generated by the columns of $A$.
 The length of vector $\vec{c}$ equals the number of equations in
 the system.
 Thus the codimension of the solution space is the difference
 between the number of equations and the number of variables.

 Now the codimension of the variety of all the operators
 that have factorizations of the type $(S_1)(S_2) \dots (S_k)$
 equals the difference between the number of equations
 and the number of variables at all the steps together.
 This can be computed using the following combinatorial fact:
\begin{lemma} \label{card} The cardinality of the set
\[
 \{ M=x_1^{d_1} \dots x_n^{d_n} \mid d_1+ \dots +d_n = t \}
\]
of monomials in $n$ independent variables $x_1, \dots, x_n$ is $
\binom{n+t-1}{t}=\binom{n+t-1}{n-1}.$
\end{lemma}
The theorem about codimension is proved.
\end{proof}
\begin{example}
 Consider all the operators of order two in two independent
 variables with symbol $S_1 \cdot S_2$,
 where $S_1, S_2$ are coprime homogeneous operators of the first
 order.
 By Theorem \ref{dimension}, the codimension of the variety of all
 the operators that have a factorization of the type $(S_1)(S_2)$,
 is $1$.

 One may find explicit formulae for the equation which defines
 this variety. Let, for example,
 $\widehat{S}_1=D_1, \ \widehat{S}_2=D_2$.
 Consider all the operators of the form
 $L=D_1 D_2 + a_{10} D_1 + a_{01} D_2+ a_{00}$.
 Such an operator has a factorization of the type $(S_1)(S_2)$
 if and only if
 coefficients $a_{10}, a_{01}, a_{00}$ satisfy the condition
\[
a_{00}- a_{10} a_{01} - \partial_x(a_{10})=0.
\]
\end{example}
\begin{example}
 Consider all the operators of order three in two independent
 variables with symbol $S_1 \cdot S_2$,
 where $S_1, S_2$ are coprime homogeneous operators of first
 and second orders respectively.
 By Theorem \ref{dimension}, the codimension of the variety of all
 the operators that have a factorization of the type $(S_1)(S_2)$,
 is $2$.

However, if we consider a factorization of the type $(S_1)(S_2)(S_3)$,
 where $S_1, S_2, S_3$ are coprime homogeneous operators of the first
 order, then, by Theorem \ref{dimension}, the codimension is $3$.
\end{example}
  To study operators that have no factorization of some type,
 we introduce the following notion:
\begin{definition}
 Let $L \in K[D]$, $\Sym_L=S_1 \dots S_k$.
 An operator $R \in K[D]$ is called a \emph{common obstacle} to
 factorization of the type $(S_1)(S_2) \dots (S_k) $ if there exists
 a factorization of this type for the operator $L-R$ and $R$ has
 minimal possible order.
\end{definition}
Common obstacles are closely related to partial factorizations:
\begin{proposition}
 Let $L \in K[D], \ \Sym_L=S_1 \dots S_k$.
 A common obstacle to a factorization of the type $(S_1)\dots(S_k)$
 is of order $t$ if and only if the minimal order of a partial
 factorization of this type is $t+1$.
\end{proposition}
 Common obstacles and their symbols are not unique in general,
 and neither of them is invariant or has some interesting properties.
 That is why we introduce the following notion.
\begin{definition} \label{ost_ring_k}
  Let $L \in K[D]$ and $\Sym_L=S_1 \cdot S_2 \cdot \dots \cdot S_k$.
  Then we say that the {\it ring of obstacles} to factorizations of the type
  $(S_1)\dots(S_k)$ is the factor ring
\[
K(S_1, \dots, S_k)= K[X] / I,
\]
 where
\[
I=\left( \frac{\Sym_L}{S_1},  \dots, \frac{\Sym_L}{S_k} \right)
\]
 is a homogeneous ideal.
\end{definition}
\begin{remark}
In the case of two factors ($k=2$), the ring of obstacles is
\[
K(S_1, S_2)= K[X] / (S_1,S_2).
\]
So the definition \ref{ost_ring_k} is a generalization of the
definition given in \cite{spain}, where we study the case of
factorizations into two factors.
\end{remark}
\begin{theorem} \label{unique_in_ring_arb_factors}
 Let $L \in K[D]$ and
 $\Sym_L=S_1 \cdot S_2 \dots  S_k,$
 where $S_i, \ i \in \{1, \dots, k\}$ are pairwise coprime.
 Then the symbols of all common obstacles to factorization
 of the type $(S_1)  \dots (S_k)$  belong to the same class in the
 factor-ring $K(S_1, \dots, S_k)$.
\end{theorem}
\begin{proof} Denote $d_i=\ord(S_i)$, $\ i \in \{1, \dots, k \}$ and
let $t$ be the order of common obstacles. In the same way as in the
proof of Theorem \ref{dimension}, we obtain the equation
(\ref{P_t=(Sym_L/S_1)}), that is the symbol of every common obstacle
can be written in the form
\[
P_{t} - ((\Sym_L/S_1) \cdot G_{t-d+d_1}^{1} + \dots + (\Sym_L/S_k)
\cdot G_{t-d+d_k}^{k}),
\]
where $P_t$ is known, uniquely determined and the same for all
common obstacles polynomial. Thus all common obstacles belong to the
class $[P_t]$ of the factor-ring $K(S_1, \dots, S_k)$.
\end{proof}
\begin{definition}
We say that the class of common obstacles in the ring of obstacles
is the \emph{obstacle to factorization}.
\end{definition}
\begin{remark}
 Every element of this class is again a common obstacle.
\end{remark}
\begin{definition} We say that two types of factorizations
  $(S_1) \dots (S_k)$ and $(b_1 S_1) \dots (b_k S_k)$
  are \emph{similar}, if $b_1, \dots, b_k \in K$ and $b_1 \dots b_k=1$.
\end{definition}
\begin{theorem} \label{Similar_types}
For an operator in $K[X]$ the rings of obstacles and the obstacles
of similar types are the same.
\end{theorem}
\begin{proof}
 Consider an operator $L \in K[D]$ and two similar types
 of factorizations of $L$:
 $(S_1) \dots (S_k)$ and
 $(b_1 S_1) \dots (b_k S_k)$, where
 $b_i \in K, \ i=1, \dots, k$.
 Then the homogeneous ideals $(S_1, \dots ,S_k)$ and $(b_1 S_1, \dots ,b_k S_k)$
 are the same, thus the rings of obstacles are also.

 Every common obstacle of the type $(S_1) \dots (S_k)$
 and of order $d_0$
 may be written as
\begin{equation} \label{obstd_0}
 P=L-(\widehat{S}_1 + T_1) \circ  \dots  \circ (\widehat{S}_k + T_k),
\end{equation}
 where $T_i$ is the sum of components of orders $d_i-1, \dots,
 d-d_i-d_0+1$, and $\ord(P)= d_0$.

 There exist $T'_1, \dots, T'_k$ such that
 $T'_i$ is the sum of components of orders  $d_i-1, \dots,
 d-d_i-d_0+1$ and
\[
(S_1 + T_1)  \circ  \dots  \circ (S_k + T_k)=(b_1 S_1 + T'_1)  \circ
\dots  \circ ( b_k S_k + T'_k).
\]
 Thus $P$ is a common obstacle of order $d_0$  of the type
 $(b_1 S_1) \dots (b_k S_k)$. On the other hand, we know that the rings of obstacles
 $K(S_1, \dots ,S_k)$ and $(b_1 S_1, \dots ,b_k S_k)$ are the same.
 Thus obstacles are the same also.
\end{proof}

 Let us recall the definition:
\begin{definition}
 A gauge transformation of $L \in K[D]$ with an
 invertible element in $g \in K$ is the operator
 $g^{-1} \circ L \circ g$.
\end{definition}
\begin{theorem} \label{conj_preserve_obst}
 Let $P$ be a common obstacle for $L \in K[D]$,
 then $g^{-1} P g$ will be a common obstacle for the gauge transformed operator $g^{-1} L g$,
 where $g \in K^*$ ($K^*$ - the set of invertible elements in $K$).
\end{theorem}
\begin{proof}
 Consider a common obstacle (\ref{obstd_0}) for $L$
 of order $d_0$. Then we have
\[
g^{-1} P g = g^{-1}L g -  g^{-1} \circ (S_1 + T_1)  \circ
\Pi_{j=2}^{k-1} (S_i+T_i)  \circ (S_k + T_k)  \circ g.
\]
 There exist $T'_1, \dots, T'_k$ such that $T'_i$ is the sum of
 components of orders $d_i-1, \dots,d-d_i-d_0+1$ and
\[
g^{-1} P g = g^{-1}L g -  (g^{-1} S_1 + T'_1)  \circ \Pi_{j=2}^{k-1}
(S_i+T'_i)  \circ (g S_k + T'_k).
\]
\end{proof}

\begin{corollary} \label{conj_preserve_obst_class}
 Obstacles are invariant under the gauge transformations.
\end{corollary}
\begin{proof}
 Under the gauge transformations common obstacles are conjugated, and so
 symbols of common obstacles are the same.
\end{proof}

\begin{theorem} \label{obstacle_d-1_zero}
 Let $n=2$, $\ L \in K[D]$, $\ord(L)=d$,
 and let $\Sym_L=S_1 \dots  S_k,$
 where $S_i, \ i \in \{1, \dots, k\}$ are pairwise coprime.
 Thus the ring of obstacles $K(S_1, \dots , S_k)$ is $0$ to order $d-1$.
 (That is, non-zero obstacles may be only less than or equal to $d-2$.)
\end{theorem}
\begin{proof}
 Denote $d_i=\ord(S_i)$, $\ i \in \{1, \dots, k \}$ and repeat
 the reasoning of the proof of Theorem \ref{dimension}.
 Thus we write equation (\ref{P_t=(Sym_L/S_1)}) for $t=d-1$:
\[
P_{d-1}=(\Sym_L/S_1) \cdot G_{d_1-1}^{1} + \dots + (\Sym_L/S_k)
\cdot G_{d_k-1}^{k}.
\]
 It has at most one solution w.r.t. $G_{d_1-1}^{1}, \dots, G_{d_k-1}^{k}$.
   Consider the corresponding system of equations in their
 coefficients. By Lemma \ref{card} the number of equations in this
 system is $d$, the number of variables is $d$ also.
   Thus the system has a unique solution,
 and so we have a partial factorization of order $d-1$.
\end{proof}

   Recall that an operator $L \in K[D]$, $\ord(L)=d$ is called strictly
 hyperbolic if the symbol of $L$ has exactly $d$ different factors.
\begin{theorem} \label{obstacle_d-2_unique}
   Let $n=2$ and $L \in K[D]$ be  strictly hyperbolic of order $d$.
 Then for each type of factorization,
 a common obstacle is unique.
\end{theorem}
\begin{proof}
    Let the type of the factorizations be $(S_1) \dots (S_d)$,
  and let $P$ be a common obstacle for this type.
  Let the order of common obstacles be $p$.
    Assume there is another common obstacle for this type,
  then it is of the form
\[
 P+ (\Sym_L/S_1) \cdot A_1+ ... + (\Sym_L/S_d) \cdot A_d,
\]
 where $A_i$ are some homogeneous polynomials of orders
 $p_i=p-\ord(\Sym_L/S_i)=p-(d-1)$.
 That is $p \geq d-1$.

 On the other hand, by Theorem \ref{obstacle_d-1_zero},
 the ring of obstacles is $0$ to order $d-1$,
 and so $p \leq d-2$.
\end{proof}
%
%
%
%
\section{Bivariate Operators of Order Two}

   Consider a second-order hyperbolic operator $L \in K[D_x,D_y]$ that is
in such a system of coordinate that the symbol of $L$ is $XY$.
   Then by Theorems \ref{obstacle_d-1_zero} and \ref{obstacle_d-2_unique}, both
 common obstacles to factorizations of $L$ have order $0$ and
 are uniquely defined.
   We compute explicit formulas.
\begin{theorem} \label{obstS1S2}
  Let
\[
L=D_x \cdot D_y+a D_x + b D_y+ c,
\]
  where $a_{10},a_{01}, a_{00} \in K$. Then obstacles of types $(X)(Y),
\ (Y)(X)$ are
\[
  \begin{array}{l}
c- ab - \partial_x(a),\\
c- ab - \partial_y(b)
  \end{array}
\]
respectively.
\end{theorem}
\begin{proof}  A factorization of $L$ of type $(X)(Y)$ has the form
\[
L=  \left(D_x+ g_{00} \right) \circ  \left( D_y +  h_{00} \right),
\]
where $g_{00}, h_{00}$ are some elements of $K$.
Comparing components
of order $1$ on the right and on the left, we have
\begin{equation} \label{L1_second_order}
(a-h_{00}) D_x + (b-g_{00})D_y=0\ ,
\end{equation}
that is $a=h_{00}$, $ b=g_{00}$. Now we compute the obstacle as
\[
L-\left(D_x+ b \right) \circ  \left( D_y + a \right)=c- ab -
\partial_x(a).
\]
One may find the obstacle for type $(Y)(X)$ analogously.
\end{proof}
\begin{remark}
  The obtained obstacles are the invariants of Laplace \cite{tsarev05}.
\end{remark}
\section{Bivariate Operators of Order Three}

Consider some operator $L \in K[D_1,D_2]$ of order three. Let the
symbol of $L$ be $S_1 \cdot S_2 \cdot S_3$, then the following types
of factorizations are possible: six types of factorization into
three factors:
\[
 (S_1)(S_2)(S_3),\ (S_1)(S_3)(S_3), (S_2)(S_1)(S_3),\ (S_2)(S_3)(S_1),\ (S_3)(S_1)(S_2),\
(S_3)(S_2)(S_1),
\]
and six types of factorization into two factors:
\[  (S_1)  (S_2  S_3),\ (S_2)  (S_1
 S_3),\ (S_3)  (S_1  S_2),\ (S_1  S_2)  (S_3),\ (S_1  S_3)
 (S_2),\ (S_2  S_3)  (S_1).
\]
\subsection{Two Factors}

  The theory introduced above applies for the case of pairwise coprime
 symbols of factors.
That is, if the considered type is $(S_1) (S_2S_3)$,
 then $S_1$ and $S_2 S_3$ should be coprime.
 Taking this and the symmetry into account,
  we restrict ourselves to considering two important special cases: factorization of the type $(X)(X^2+XY)$ for an
 operator with symbol $X^2Y+XY^2$ and of the type $(X)(Y^2)$ for an
 operator with symbol $XY^2$.

  Note that by Theorem \ref{obstacle_d-1_zero} common obstacles of
 these types may be of orders one and zero only,
 in the first case a common obstacle is not unique.
\begin{theorem} \label{obs_L_3}
  Let
\[
L=\widehat{\Sym}_L + a_{20}D_{xx}+ a_{11}D_{xy}+ a_{02}D_{yy}+
a_{10}D_{x}+ a_{01}D_{y}+ a_{00},
\]
  where all $a_{ij} \in K.$

  Let $\Sym_L=XY(X+Y)$, then
\[
\begin{array}{ll}
 \Obst_{(X)(YX+YY)}=&\Big(a_{02}^2- a_{11} a_{02}+a_{01}+\partial_x(a_{02}-a_{11}) \Big)
                     D_y +\\
& a_{00} - a_{02} a_{10} + a_{02}^2 a_{20} + 2 a_{02}
\partial_x(a_{20})-\partial_x(a_{10})+ a_{20} \partial_x(a_{02})+\partial_{xx}(a_{20}),
\end{array}
\]
is a common obstacle to factorizations of $L$ of type $(X)(YX+YY)$.

  Let $\Sym_L=X^2Y$, then
\[
\begin{array}{ll}
 \Obst_{(Y)(XX)}=&\Big(a_{10} - a_{20}a_{11} - \partial_y(a_{11}) \Big)
                     D_x +\\
& a_{00} - a_{20} a_{01} + a_{20}^2 a_{02} + 2 a_{20}
\partial_y(a_{02})-\partial_y(a_{01})+  a_{02}\partial_y(a_{20})+\partial_{yy}(a_{02}),
\end{array}
\]
is a common obstacle to factorizations of $L$ of type $(Y)(XX)$.
\end{theorem}
\begin{proof}
 All factorizations of type $(X)(YX+YY)$ have the form
\begin{equation} \label{L_S_1_S_2S_31}
L= \left(D_x + G_0 \right) \circ  \left( D_{xy} + D_{yy}+ H_1 + H_0
\right),
\end{equation}
where $G_0=g_{00} \in K$, $H_1=h_{10} D_x + h_{01} D_y \in
K[D_x,D_y]$, $H_0=h_{00} \in K$. Compare components of order $2$ on
 both sides of equality (\ref{L_S_1_S_2S_31}), then we get a
system of linear equations in coefficients $h_{10}, h_{01}, g_{00}$:
\[
  \begin{cases}
    a_{20} = h_{10}, \\
    a_{11} = h_{01} +  g_{00}, \\
    a_{02} = g_{00}.
  \end{cases}
\]
We find the unique solution of the system. Then, we compare
coefficients in $D_x$ on the both sides of (\ref{L_S_1_S_2S_31}),
and so we get
\[
h_{00}=a_{10}-a_{20}a_{02}- \partial_x(a_{20}).
\]
Now we may compute a common obstacle as $P=L- \left(D_x + G_0
\right) \circ  \left( D_{xy} + D_{yy}+ H_1 + H_0 \right)$.

One may find the obstacle for type $(Y)(XX)$ analogously.
\end{proof}
\subsection{Three Factors}

Here it is enough to consider the case of hyperbolic operators with
symbol $XY(X+Y)$ and type $(X)(Y)(X+Y)$ of factorizations. In this
case a common obstacle may be of orders $1$ and $0$ only (Theorem
\ref{obstacle_d-1_zero}) and it is unique (Theorem
\ref{obstacle_d-2_unique}).

\begin{theorem} \label{obs_L_3_three_factors}
Let
\[
L=D_x D_y (D_x+D_y) + a_{20}D_{xx}+ a_{11}D_{xy}+ a_{02}D_{yy}+
a_{10}D_{x}+ a_{01}D_{y}+ a_{00},
\]
where all $a_{ij} \in K.$ The common obstacle of type $(X)(Y)(X+Y)$
is
\[
\begin{array}{ll}
 \Obst_{(X)(Y)(X+Y)}=&(a_{10}-a_{20}a_{11}+a_{20}^2-\partial_x(a_{20})+\partial_y(s_2))
D_x+\\
 &(a_{01}-a_{02}a_{11}+a_{02}^2+\partial_x(-a_{11}+a_{02})) D_y+\\
 &a_{00} + a_{20}a_{02}s_2+ s_2 \partial_x(a_{20})+ (a_{20} \partial_x+
\partial_{xy}+a_{02}\partial_y)(s_2),
\end{array}
\]
where $s_2=a_{20}-a_{11}+a_{02}$.
\end{theorem}
\begin{proof}
Every factorization of type $(X)(Y)(X+Y)$ has the form:
\begin{equation} \label{L_S_i_S_jS_k}
L= \left(D_x+ g_0 \right) \circ  \left( D_y  + h_0 \right) \circ
\left( D_x + D_y + f_0 \right).
\end{equation}
Compare components of order $2$ and get the only solution
\[
h_0=a_{20}, \ g_0=a_{02}, \ f_0=a_{11}-a_{02}-a_{20}.
\]
Now, we may compute the  common obstacle as the difference of the
left and the right sides of the equation (\ref{L_S_i_S_jS_k}).
\end{proof}


\begin{thebibliography}{99}

\bibitem{GS} D.Grigoriev, F.Schwarz {\it Factoring and Solving Linear Partial Differential Equationse},
In {\it J. Computing 73, pp.179-197} (2004)

\bibitem{blumberg} H. Blumberg. \textit{ \"Uber algebraische Eigenschaften von linearen
homogenen Differentialausdr\"ucken.} Diss., G\"ottingen. 1912.

\bibitem{kart}  E. Kartashova {\it Hierarchy of general invariants for bivariate LPDOs},
J. { \it Theoretical and Mathematical Physics}, 2006.

\bibitem{maple} www.maple.com.

\bibitem{invariants} E. Shemyakova,
{\it A Full System of Invariants for Third-Order Linear Partial Differential Operators},
J. \emph{Lecture Notes in Computer Science}, Springer, 2006.

\bibitem{spain} E. Shemyakova, F. Winkler,
{\it Obstacle to Factorization of LPDOs },
 {\it Proc. Transgressive Computing}, Granada, Spain, 2006.

\bibitem{tsarev94} S.P. Tsarev, {\it On the problem of factorization of linear ordinary
differential operators}, {\it Programming \& computer software}, v.
20, \# 1, pp. 27--29, 1994.

\bibitem{tsarev05} S.P. Tsarev. {\it Generalized Laplace Transformations and Integration of
Hyperbolic Systems of Linear Partial Differential Equations}, proc.
ISSAC'05. 2005.

\bibitem{vinogradov} A. Vinogradov, I. Krasilshik, B. Lychagin {\it Introduction in geometry
of nonlinear differential systems} (in russian), Nauka, (1986).

\end{thebibliography}
\end{document}